\documentclass{amsart}
\usepackage{amssymb, mathrsfs, comment}
\usepackage{color}

\newtheorem{prop}{Proposition}
\newtheorem{lem}[prop]{Lemma}
\newtheorem{thm}[prop]{Theorem}

\newtheorem{cor}[prop]{Corollary}

\theoremstyle{definition}
\newtheorem{defn}[prop]{Definition}

\newcommand{\ddb}{\sqrt{-1}\partial\overline{\partial}}
\newcommand{\Bl}{\mathrm{Bl}}

\title[]{Extremal metrics on blowups along submanifolds}
\author[R. Seyyedali]{Reza Seyyedali}
\author[G. Sz\'ekelyhidi]{G\'abor Sz\'ekelyhidi}

\thanks{G.Sz. is supported in part by NSF grants DMS-1306298 and DMS-1350696.}

\address{Department of Mathematics, Boyd Graduate Studies Research
  Center, University of Georgia, Athens, GA 30602}
\email{rseyyeda@uga.edu}

\address{Department of Mathematics, University of Notre Dame, Notre
  Dame, IN 46556}  
\email{gszekely@nd.edu}

\begin{document}
\begin{abstract}
  We give conditions under which the blowup of an extremal K\"ahler manifold along a
  submanifold of codimension greater than two admits an extremal
  metric. This generalizes work of Arezzo-Pacard-Singer, who
  considered blowups in points. 
\end{abstract}
\maketitle

\section{Introduction}
A basic question in K\"ahler geometry is the existence of extremal 
metrics on K\"ahler manifolds, in the sense of Calabi~\cite{Cal82}. A
K\"ahler metric $\omega_M$ on $M$ is an extremal metric if the gradient $\nabla S(\omega_M)$
of its scalar curvature is a holomorphic vector field on $M$. The
Yau-Tian-Donaldson conjecture~\cite{Yau, Tian97, DonToric, GSzExt} relates
the existence of an extremal metric on a compact K\"ahler manifold to
an algebro-geometric stability condition, but so far there are only a
few existence results beyond the K\"ahler-Einstein case~\cite{Yau78,
  Tian97, CDS}. 

In this
paper, following the works of Arezzo-Pacard~\cite{AP06, AP09},
Arezzo-Pacard-Singer~\cite{APS06} and the second author~\cite{GSz10, GSz13_1} we
investigate the existence of an extremal metric on a blowup $\Bl_S M$
of $M$ along a smooth submanifold, assuming that $M$ admits an
extremal metric $\omega_M$. The main new feature in our work is that we
 allow $\dim S > 0$, while previous works focused
on blowups in points with the exception of Hashimoto~\cite{Ha15} who considered blowups of projective spaces in lines.  

In order to state our result we set up some notation. We suppose that
$S$ is a codimension-$k$ submanifold of $M$, and we write
$G$ for the group of Hamiltonian isometries of
$(M,\omega_M)$. There is an associated moment map
\[ \mu : M \to \mathfrak{g}^*, \]
normalized so that $\langle\mu, \xi\rangle$ has zero integral for each
$\xi\in\mathfrak{g}$. Denoting by $\mathscr{S}$ the space of
codimension-$k$ complex submanifolds of $M$, the group $G$ acts on
$\mathscr{S}$, preserving a natural symplectic form, and 
we have a moment map 
\[ \begin{aligned} \mu_{\mathscr{S}} : \mathscr{S} &\to \mathfrak{g}^*, \\
 S &\mapsto \int_S \mu\,\omega_M^{n-k}. 
\end{aligned}\]

We identify $\mathfrak{g}=\mathfrak{g}^*$ using the
$L^2$-product on Hamiltonian functions, and so we can naturally think
of $\mu_{\mathscr{S}}(S)$ as a vector field on $M$. 
In analogy with the result in \cite{GSz10}, in this paper we prove
the following. 
\begin{thm}\label{thm:extremal}
  Suppose that $S\in\mathscr{S}$ is a submanifold such that $\nabla
  S(\omega_M)$ and the vector
  field $\mu_{\mathscr{S}}(S)$ are tangent to $S$. Assume also that the
  codimension of $S$ is $k > 2$. Then $\Bl_S M$ admits
  an extremal metric in the class $[\omega_M] - \epsilon^2[E]$ for
  sufficiently small $\epsilon > 0$. 
\end{thm}

The strategy of the proof is very similar to that employed in
\cite{GSz10}.
Because of technical difficulties we have not obtained the same result
when $k=2$, although it is very likely that it is true in that
case as well. 

Our result can be used to obtain many new examples of extremal
metrics. The simplest situation is when
$(M,\omega_M)$ has trivial isometry group, and so in particular
$\omega_M$ has constant scalar curvature. In this case the moment map
$\mu_{\mathscr{S}}$ is trivial, and so for any submanifold $S\subset
M$ of codimension greater than two the blowup $\mathrm{Bl}_S M$ admits
a constant scalar curvature metric in $[\omega_M] - \epsilon^2[E]$ for
small $\epsilon$. A more general result, allowing for a non-trivial
automorphism group, analogous to \cite[Theorem 2.4]{APS06}, is the
following. 

\begin{cor}
  Suppose that $\omega_M$ is an extremal metric on $M$, and let $T$ be
  a maximal torus in the isometry group of $(M,\omega_M)$. Suppose
  that $S\subset M$ has codimension greater than 2, and the action of
  $T$ preserves $S$. Then $\mathrm{Bl}_S M$ admits an extremal metric
  in $[\omega_M] - \epsilon^2 [E]$ for sufficiently small $\epsilon >
  0$. 
\end{cor}
\begin{proof}
  The vector field $J\nabla S(\omega_M)$ is invariant under the adjoint
  action of the isometry group of $\omega_M$, and so it is in the
  center of the Lie algebra $\mathfrak{g}$ of the isometry group
  $G$. In particular $J\nabla S(\omega_M) \in \mathfrak{t}$, where
  $\mathfrak{t}$ is the Lie algebra of $T$. Similarly because the
  moment map $\mu_{\mathscr{S}}$ is equivariant,
  $\mu_{\mathscr{S}}(S)$ is in the center of the stabilizer of $S$
  under the infinitesimal action of $\mathfrak{g}$. By our assumption
  this stabilizer contains $\mathfrak{t}$, so any element in its
  center must belong to $\mathfrak{t}$. In particular
  $\mu_{\mathscr{S}}(S) \in \mathfrak{t}$ (here we are identifying
  $\mathfrak{g} \cong \mathfrak{g}^*$ as before), and so
  $\mu_{\mathscr{S}}(S)$ fixes $S$. Theorem~\ref{thm:extremal} then applies.
\end{proof}

This corollary applies for example to subspaces $\mathbf{P}^k \subset
\mathbf{P}^n$ as long as $n > k+2$. In this way we obtain some
extensions of the work of Hashimoto~\cite{Ha15}, who showed that $\Bl_{\mathbf{P}^1}
\mathbf{P}^n$ admits an extremal metric for all $n$.  More generally we can let $M$
be any toric manifold which admits an extremal metric, for instance a
K\"ahler-Einstein metric obtained using the existence result of
Wang-Zhu~\cite{WZ04}. We can then choose $S\subset M$ to be a toric
submanifold of codimension greater than 2. 

There are also more general submanifolds $S\subset \mathbf{P}^n$ satisfying the
assumption in Theorem~\ref{thm:extremal} that $\mu_{\mathscr{S}}(S)$
is tangent to $S$. The condition $\mu_{\mathscr{S}}(S)=0$ means that
$S\subset \mathbf{P}^n$ is a balanced embedding, and
Donaldson~\cite{Don01} showed that if $\mathrm{Aut}(S)$ is trivial,
and $S$ admits a constant scalar curvature metric, then there are
balanced embeddings $S \subset \mathbf{P}^N$ for sufficiently large
$N$. This result was generalized by the first named author~\cite{Sey} to
the case when $S$ has non-trivial automorphisms, and admits an
extremal metric (see also Mabuchi~\cite{Ma1, Mab09},
Hashimoto~\cite{Ha15B} for other work in this direction). 
As a consequence we have the following. 

\begin{cor} Let $(S,\omega_S)$ be an extremal K\"ahler manifold, with
  $\omega_S\in c_1(L)$ for a line bundle $L\to S$. Fix an integer $r >
  0$, and an embedding $S\subset \mathbf{P}^N$ using a basis of
  sections of $L^r$. If $r$ is sufficiently large, then $\mathrm{Bl}_S
  \mathbf{P}^N$ admits an extremal metric in the class $[\omega_{FS}]
  - \epsilon^2[E]$ for small $\epsilon$. 
\end{cor}
\begin{proof}
  Theorem 1.1 in \cite{Sey} implies that under the assumptions there
  exist relatively balanced embeddings $S\subset \mathbf{P}^N$ using
  a basis of sections of $L^r$ for sufficiently large $r$. This means
  precisely that under these embeddings $\mu_{\mathscr{S}}(S)$, 
  identified with a vector field on
   $\mathbf{P}^N$, is tangent to $S$. Our main result,
  Theorem~\ref{thm:extremal}, then implies the required result. 
\end{proof}

The structure of the paper is the following. In
Section~\ref{sec:approx1} we will write down a metric
$\omega_\epsilon$ on $\Bl_SM$ giving a first approximation to the
metric that we are looking for, and we will present the
main gluing result that we need to prove. In
Section~\ref{sec:linearized} we will show that the linearized operator
of our problem is invertible. We will complete the proof of our main result
in Section~\ref{sec:nonlinear} by constructing a better approximate
solution $\widetilde{\omega}_\epsilon$, in an analogous way to what was done in
\cite{APS06, GSz10}, and then controlling the relevant
non-linear terms.

\section{The gluing problem}\label{sec:approx1}

\subsection{A first approximate solution}
Suppose as in the introduction 
that $(M,\omega_M)$ is a compact K\"ahler manifold such that
$\omega_M$ is an extremal metric. Let $S\subset M$ be a codimension $k$
submanifold, where $k > 2$. Our goal in this section is to construct a
K\"ahler metric $\omega_\epsilon$ on the blowup $\Bl_SM$ in the class
$[\omega_M] - \epsilon^2[E]$ for sufficiently small $\epsilon$,
which will be a first approximation to the extremal metric that we
seek. In previous
work by Arezzo-Pacard~\cite{AP06},
Arezzo-Pacard-Singer~\cite{APS06} and the second author~\cite{GSz10},
$S$ was a point, and the approximate solution on $\Bl_SM$ was constructed by
identifying an annulus around $S$ by an annulus inside the blowup
$\Bl_0\mathbf{C}^n$. When $S$ is a submanifold, there is no longer a
standard form of a neighborhood of $S$, and so we will instead
view $\Bl_SM$ as a completion of $M\setminus S$ under a suitable
metric. In other words we will identify the complement of the
exceptional divisor in $\Bl_SM$ with $M\setminus S$, and our
constructions will primarily take place on $M\setminus S$. We then
simply need to ensure that our metric extends to a smooth metric on
$\Bl_SM$, which we will achieve by using the usual coordinate charts
covering the blowup. 

The basic building block for constructing extremal metrics on blowups
is the Burns-Simanca metric~\cite{Sim91} on $\Bl_0\mathbf{C}^k$. This is
 a scalar flat, asymptotically flat K\"ahler metric
\[ \eta = \ddb\Big(|w|^2 + \gamma(|w|) \log |w|^2 + \psi(|w|^2)\Big), \]
where $\psi : [0,\infty)$ is smooth up to the boundary, and 
 $\psi$ is in the weighted space $C^\infty_{2-k}$, i.e. $\nabla^i
\psi(t) = O(t^{2-k-i})$ for all $i$, as $t\to \infty$. In addition
$\gamma:\mathbf{R}\to\mathbf{R}$ is a cutoff function such that
$\gamma(t) = 1$ for $t<1$ and $\gamma(t) = 0$ for $t > 2$. There are
also more refined expansions of $\psi$. We will need to use that
(see e.g. \cite[Lemma 26]{GSz10})
\[ \psi(|w|^2) = |w|^{4-2k} + \widetilde{\psi}(|w|^2), \]
where $\widetilde{\psi} \in C^\infty_{1-k}$. 

We will define $\omega_\epsilon$ by using the
K\"ahler potential of the Burns-Simanca metric, but replacing $|w|$ by
the distance function $d$ to the submanifold $S$, with respect to the
metric $\omega_M$. 
Note that $d^2$ is a smooth function in a tubular neighborhood of
$S$. For small $\epsilon > 0$ let us define $r_\epsilon =
\epsilon^\alpha$ for
\[ \alpha = \frac{2k}{2k+1}. \]
In addition let $\gamma:[0,\infty) \to \mathbf{R}$ be a cutoff
function as above, and define $\gamma_2 : M\to \mathbf{R}$ by $\gamma_2 =
\gamma(r_\epsilon^{-1}d)$, and $\gamma_1 = 1-\gamma_2$. So $\gamma_1$
is supported away from $S$, while $\gamma_2$ is supported near $S$. 
 
Finally we define
\[ \omega_\epsilon = \omega_M + \epsilon^2\ddb\Big( \gamma_2 \big[
\gamma(\epsilon^{-1}d) \log (\epsilon^{-2}d^2) +
\psi(\epsilon^{-2}d^2)\big]\Big), \]
on $M\setminus S$. 

\begin{prop}\label{prop:definesmetric} For sufficiently small $\epsilon$, the form
  $\omega_\epsilon$ defines a K\"ahler metric on $M\setminus S$,
  extending to a smooth metric on $\Bl_SM$. 
\end{prop}

In the proof we will need the following.
\begin{lem}\label{lem:charts}
  At any point $p\in S$, we can choose coordinates $z_1,\ldots, z_k, w_1,\ldots,
  w_{n-k}$, defined for $|z|, |w| < 1$, such that $S = \{ z_i=0\}$ and
  \[ d^2 = |z|^2(1+\rho(z,w)), \]
  where $\rho = O(|z| + |w|)$, and all derivatives of $\rho$ are
  bounded. In addition 
  \[ \omega_M = \ddb( |z|^2 + |w|^2 + \phi(z,w)), \]
  where $\nabla^i \phi = O(|z|^{3-i} + |w|^{3-i})$ for $i < 3$, while
  higher order derivatives are bounded. All of these bounds can be chosen to be
  uniform in the point $p$. 
\end{lem}

\begin{proof}[Proof of Proposition~\ref{prop:definesmetric}]
  We will work on four separate regions: 
\begin{itemize}
\item
  On the set where $d > 2r_\epsilon$, we have $\omega_\epsilon =
  \omega_M$, so it is a smooth metric. 
\item  On the set $r_\epsilon/2 < d < 4r_\epsilon$, we have
  contributions from the derivatives of $\gamma_2$, but the term 
  involving $\log$ is not present. The asymptotics of $\psi$ imply
  that, measured with respect to the metric $\omega_M$, we have
  \[ \Vert \nabla^i(\omega_\epsilon - \omega_M) \Vert =
  O(\epsilon^{2k-2}r_\epsilon^{2-2k-i}), \]
  for all $i$.
  It follows that for sufficiently small $\epsilon$, the form $\omega_\epsilon$
  is also positive. 
\item
  On the set $2\epsilon < d < r_\epsilon$ we have $\gamma_2=1$ and
  $\gamma(\epsilon^{-1}d)=0$. We change coordinates, using
  Lemma~\ref{lem:charts}.  In terms of $z_1,\ldots, z_k,
w_1,\ldots, w_{n-k}$ above, we set
  \[ \begin{aligned} Z_i &= \epsilon^{-1}z_i \\
        W_j &= \epsilon^{-1}w_j.
\end{aligned} \]
By shifting the center of the coordinate system we can assume that
$|W| < 1$. 
We can compare $\epsilon^{-2}\omega_\epsilon$ with the product metric
on $\mathrm{Bl}_0 \mathbf{C}^k \times \mathbf{C}^{n-k}$. In our
coordinates a K\"ahle potential for $\epsilon^{-2}\omega_\epsilon$ is
given by
\[ \begin{aligned}
  F &= \sum_i |Z_i|^2 + \sum_j |W_j|^2 + \epsilon^{-2} \phi(\epsilon
  Z_i, \epsilon W_j) \\
&\quad + \psi(\epsilon^{-2}d^2),
\end{aligned} \]
where $d$ is given by
\[ \epsilon^{-2}d^2 = \left( \sum_i |Z_i|^2\right) (1 + \rho(\epsilon
Z_i, \epsilon W_j)). \]
At the same time the product metric has K\"ahler potential
\[ F_{prod} = \sum_i |Z_i|^2 + \sum_j |W_j|^2 + \psi\left( \sum_i
  |Z_i|^2\right), \]
and this product metric is uniformly equivalent to the Euclidean metric in the
$Z_i, W_j$ coordinates. 
We have
\[ F - F_{prod} = \epsilon^{-2}\phi(\epsilon Z_i, \epsilon W_j) +
\psi(\epsilon^{-2}d^2) - \psi(|Z|^2). \]
The estimates we have for the derivatives of $\phi$ imply, that in the
$Z_i, W_j$ coordinates
\[ \nabla^i \epsilon^{-2}\phi(\epsilon Z_i, \epsilon W_j) =
O(\epsilon^{-2}\epsilon^i |\epsilon Z|^{3-i}) = O(\epsilon
|Z|^{3-i}). \]
At the same time, we have
\[ \nabla^i(\epsilon^{-2}d^2 - |Z|^2) = O(\epsilon |Z|^{3-i}), \]
and so the decay estimates for $\psi$ imply that
\[ \nabla^i(\psi(\epsilon^{-2}d^2) - \psi(|Z|^2)) = O(\epsilon |Z|^{5-2k-i}).\]
Since $|Z| > 1$ and $5-2k < 3$, we have
\begin{equation}\label{eq:F1} \nabla^i(F - F_{prod}) =
  O(\epsilon^{i-2} d^{3-i})),
\end{equation}
using also that $d$ is comparable to $\epsilon|Z|$. 
For small $\epsilon$ the form $\epsilon^{-2}\omega_\epsilon$ will then
be a small perturbation of the product metric, since on this region $d
\ll 1$. 

\item
  Finally, to examine the set where $d < 2\epsilon$ we perform a
  different change of coordinates.
  In terms of $z_1,\ldots, z_k,
w_1,\ldots, w_{n-k}$ above, we set
\[ \begin{aligned} v &= \epsilon^{-1}z_k, u_i = \frac{z_i}{z_k}, \text{ for } i=1,
\ldots, k-1 \\
w'_j &= \epsilon^{-1}w_j\, \text{ for } j=1,\ldots, n-k. 
\end{aligned}\]
We use this chart at points where $|z_k| > \frac{1}{2}\max\{|z_1|,
\ldots, |z_{k-1}|\}$, say, and permute the coordinates appropriately
at other points. On our region, once
$\epsilon$ is sufficiently small, we have $\frac{d}{2} < |z| < 2d$. It
follows that $|v| < 2$, $|u_i| < 2$, and by changing the basepoint
for the coordinate system, we can assume that $|w'| < 1$. In these
coordinates we have
\[ d^2 = \epsilon^2 |v|^2\left( 1+ \sum_i |u_i|^2\right)(1 +
\rho(\epsilon vu_i, \epsilon v, \epsilon w')). \]

Once again we will see
 that in these coordinates $\epsilon^{-2}\omega_\epsilon$ is well
approximated by the product metric on $\Bl_0\mathbf{C}^k \times
\mathbf{C}^{n-k}$. Indeed, in these coordinates a K\"ahler potential
for $\epsilon^{-2}\omega_\epsilon$ is given by
\[ \begin{aligned}
 F &= |v|^2 + \sum_i |v|^2|u_i|^2 + \sum_j |w'_j|^2 +
  \epsilon^{-2}\phi(\epsilon vu_i, \epsilon v, \epsilon w') \\
&\quad + \gamma(\epsilon^{-1}d) \log |v|^2 \left(1 + \sum_i
  |u_i|^2\right) \\
&\quad + \gamma(\epsilon^{-1}d) \log \left(1 + \rho(\epsilon vu_i,
  \epsilon v, \epsilon w')\right) + \psi(\epsilon^{-2}d^2),
\end{aligned}\]
while a K\"ahler potential for the product metric is
\[ \begin{gathered}
 F_{prod} = |v|^2 + \sum_i |v|^2 |u_i|^2 + \sum_j |w'_j|^2 +
 \gamma(D)\log D^2  +\psi(D^2), 
\end{gathered}\]
where
\[ D^2 = |v|^2\left(1 + \sum_i |u_i|^2\right). \]

On our region we have $D < 4$, and the derivatives of $\psi$ are all
bounded. It follows that 
\begin{equation}\label{eq:nablaiF} 
\nabla^i(F - F_{prod}) = O(\epsilon), 
\end{equation}
for all $i\geq 2$, where we are taking derivatives in the $u_i, v,
w_k'$ coordinates. In particular, once $\epsilon$ is sufficiently
small, $\epsilon^{-2}\omega_\epsilon$ will define a smooth metric uniformly
equivalent to the product metric on this set.
\end{itemize}
\end{proof}

\subsection{The gluing result}
The overall strategy to proving Theorem~\ref{thm:extremal} 
is the same as in \cite{GSz10}. We first
choose a maximal torus $T$ in the stabilizer $G_S$ of the submanifold
$S$, and work $T$-equivariantly throughout. Let $H\subset G$ denote the
centralizer of $T$, and $\overline{\mathfrak{h}}$ the Hamiltonian
functions corresponding to the Lie algebra of $H$ (including the
constants). So $\dim\overline{\mathfrak{h}} = \dim H + 1$. Write
$\overline{\mathfrak{t}}\subset \overline{\mathfrak{h}}$ for the
subspace corresponding to $T$. The elements in
$\overline{\mathfrak{t}}$ lift to the blowup $\Bl_S M$ in a natural
way, giving Hamiltonians of holomorphic vector fields with respect to
the metric $\omega_\epsilon$. For this note that the function $d$ is invariant under the
action of $T$, and so $\omega_\epsilon$ is $T$-invariant. 

In \cite{GSz10} we defined a lifting of the rest
of the functions in $\overline{\mathfrak{h}}$ using cutoff functions,
but here we proceed in a slightly different way, simply pulling back
the functions under the blowdown map $\Bl_S M\to M$. 

\begin{defn}\label{defn:lifting}
  We define a map 
  \[ \mathbf{l} : \overline{\mathfrak{h}} \to C^\infty(\Bl_SM), \]
  depending on $\epsilon$, in the following way. We fix a decomposition
  $\overline{\mathfrak{h}} = \overline{\mathfrak{t}} \oplus
  \mathfrak{h}'$. We lift elements $g\in \overline{\mathfrak{t}}$ to
  $\Bl_S M$ in the natural way: if we write $\omega_\epsilon = \omega_M + \ddb A$, then
  on $M\setminus S$ we have
  \[ \mathbf{l}(g) = g + \frac{1}{2}X(A), \]
  where $X = \nabla g$ is the holomorphic vector field corresponding
  to $g$. This function extends to give a smooth function on $\Bl_S
  M$, and it is the Hamiltonian, with respect to $\omega_\epsilon$, of
  the vector field $JX$. 

  For $g\in \mathfrak{h}'$ we simply define $\mathbf{l}(g) = g$, and
  note that this also defines a smooth function on $\Bl_S M$, since
  the blowdown map $\Bl_S M\to M$ is smooth. 
\end{defn}

Given this definition, the gluing result that we need to show is the
following.

\begin{prop}\label{prop:gluing}
  Suppose that $S\in \mathscr{S}$ is such that $\nabla S(\omega_M)$
  is tangent to $S$. There are constants $\epsilon_0, c > 0$ such that
  for all $\epsilon < \epsilon_0$ we can find $u\in C^{\infty}(\Bl_S
  M)^T$ and $f\in \overline{\mathfrak{h}}$ satisfying
  \begin{equation}\label{eq:maineq}
 S(\omega_\epsilon + \ddb u) - \frac{1}{2}\nabla\mathbf{l}(f)\cdot
  \nabla u  = \mathbf{l}(f). 
\end{equation}
  In addition we have an expansion
  \[ f = S(\omega_M) + \epsilon^{2k-2}(\lambda +
  c_m\mu_{\mathscr{S}}(S)) + f_\epsilon, \]
  where $c_m, \lambda$ are constants, and $|f_\epsilon| \leq
  c\epsilon^\kappa$ for some $\kappa > 2k-2$. 
\end{prop}

Based on this proposition, the proof of Theorem~\ref{thm:extremal} is
identical to the argument in \cite[p. 1426]{GSz10}. For the reader's
convenience we give the main points here. 

\begin{proof}[Proof of Theorem~\ref{thm:extremal}]
We are assuming that $\nabla S(\omega_M)$ and $\mu_{\mathscr{S}}(S)$
are tangent to $S$. We choose our maximal torus $T$ so that
$\mu_{\mathscr{S}}\in \overline{\mathfrak{t}}$. Note that we also have
$S(\omega_M)\in \overline{\mathfrak{t}}$ because $J\nabla S(\omega_M)$ 
is in the center of $G$. The complexification $H^c$ of the group $H$
acts on the space $\mathscr{S}$, and we want to show that for
sufficiently small $\epsilon$ we can find an element $h\in H^c$ near
the identity, so that Proposition~\ref{prop:gluing} applied to the
perturbed submanifold $h\cdot S$ yields an extremal metric on
$\Bl_{h\cdot S} M$.

The key point for this is that $f$, as a map from submanifolds to $\overline{\mathfrak{h}}$ can be viewed as 
a perturbation of a moment map, 
and so \cite[Proposition 8]{GSz10} can be applied. We obtain a small
perturbation $h\cdot S$ of $S$, such that when
Proposition~\ref{prop:gluing} is applied at $h\cdot S$, then the
vector field induced by $f$ is tangent to $h\cdot S$. In particular
the metric $\omega_\epsilon + \ddb u$ constructed in
Proposition~\ref{prop:gluing} will then be an extremal metric on
$\Bl_{h\cdot S} M$. At the same time, $\Bl_{h\cdot S} M
\cong \Bl_S M$, and so we obtain the required extremal metric on
$\Bl_S M$. 
\end{proof}

\section{The linearized problem}\label{sec:linearized}
In this section we study the linearized problem corresponding to
Equation~\eqref{eq:maineq}. Let us denote by $C^\infty(\Bl_S M)^T_0$
the $T$-invariant functions $u$ on $\Bl_SM$ satisfying $\langle u,
f\rangle = 0$ for all $f\in \overline{\mathfrak{h}}$, where the inner
product is computed using $\omega_\epsilon$. We will consider the
linear operator
\[ \begin{aligned} \widetilde{L} : C^\infty(\Bl_S M)^T_0 \times
  \overline{\mathfrak{h}} &\to C^\infty(\Bl_S M)^T,  \\
  (u, f) &\mapsto L_{\omega_\epsilon}(u) - \frac{1}{2}
  \nabla\mathbf{l}(\mathbf{s})\cdot \nabla u -
  \mathbf{l}(f). \end{aligned} \]
Here $L_{\omega_\epsilon}$ denotes the linearization of the scalar curvature operator at
$\omega_\epsilon$, i.e. 
\[ S(\omega_\epsilon + \ddb u) = S(\omega_\epsilon) + L_{\omega_\epsilon}(u) + Q_{\omega_\epsilon}(u), \]
for a suitable non-linear operator $Q_{\omega_\epsilon}$, and
$\mathbf{s} = S(\omega_M)$. Recall (see e.g. \cite[Section
4.1]{GSzBook}) that we have
\[ L_{\omega_\epsilon}(u) = -\Delta^2 u - R^{\bar
  kj}\partial_j\partial_{\bar k}u, \]
in terms of the Ricci curvature $R_{j\bar k}$ of $\omega_\epsilon$. In
addition we will need to relate this to the Lichnerowicz operator
$\mathcal{D}^*\mathcal{D}$, where $\mathcal{D}u = \bar{\partial}
\nabla^{1,0}u$. We have
\[ \mathcal{D}^*\mathcal{D}u = \Delta^2 u + R^{\bar
  kj}\partial_j \partial_{\bar k}u + \frac{1}{2}\nabla
S(\omega_\epsilon)\cdot \nabla u. \]

We will show that the operator $\widetilde{L}$ is invertible, and that
we can control the norm of its inverse in suitable weighted spaces. 

\subsection{Weighted spaces}
We will next define the weighted H\"older spaces that we will use. Let us define the
weight function $\tau : M\to \mathbf{R}$ by
\[ \tau(x) = \begin{cases} 1&\, \text{ if } d(x) \geq 1 \\
  d(x) &\, \text{ if } \epsilon \leq d(x)\leq 1 \\
  \epsilon &\, \text{ if } d(x)\leq \epsilon, \end{cases}\]
and extend it to $\Bl_SM$ by continuity.  
We define the weighted space $C^{l,\alpha}_\delta$ on $\Bl_SM$,
depending on $\epsilon$, as follows. The estimate $\Vert f\Vert_{C^{l,\alpha}_\delta} <
c$ means that for any $p \in \Bl_SM$ the $C^{l,\alpha}$-norm of $f$ on
an $\omega_M$-ball of radius $\tau(p) / 10$ around $p$ is bounded by
$c\tau(p)^{\delta}$, measured with respect to the scaled up metric
$\tau(p)^{-2}\omega_\epsilon$. 

In practice we can control these weighted norms as follows. On
the region where $d > r_\epsilon$, 
the metric $\omega_\epsilon$ is
uniformly equivalent to $\omega_M$, and so on this region $\Vert
f\Vert_{C^l_\delta} < c$ means that
\[ |\nabla^i f| < c' \tau^{\delta -i}, \]
for $i\leq l$, 
with the derivatives measured using $\omega_M$. On the region
where $d < 2r_\epsilon$, as we have seen in the proof of
Proposition~\ref{prop:definesmetric},
 the scaled up metric $\epsilon^{-2}\omega_\epsilon$ is uniformly equivalent to the product
metric on $\Bl_0\mathbf{C}^k \times \mathbf{C}^{n-k}$. This in turn
is uniformly equivalent, on suitable charts, with the Euclidean metric in
terms of our coordinates $Z_i, W_j$, or $v, u_i, w_j'$ from the proof of Proposition~\ref{prop:definesmetric}. It follows that in these
coordinates $\Vert f\Vert_{C^l_\delta} < c$ means either
\begin{equation}
 |\nabla^i f| < c'\epsilon^\delta |Z|^{\delta - i} 
\end{equation}
in the $Z_i, W_j$ charts (where $\tau$ is comparable to $\epsilon|Z|$), or
\[  |\nabla^i f| < c' \epsilon^\delta, \]
in the $v, u_i, w_j'$ charts. 

We also define analogous weighted spaces on $\mathbf{C}^{n-k} \times
\Bl_0 \mathbf{C}^k$ using the weight function given by $\tau(x) = 1 +
d(x)$, where $d(x)$ is distance from the exceptional divisor in
$\Bl_0\mathbf{C}^k$, and weighted spaces on
$(\mathbf{C}^k\setminus\{0\})\times \mathbf{C}^{n-k}$ using
the weight function $\tau(x) = |z|$ in terms of the coordinate $z$ on
the $\mathbf{C}^k\setminus \{0\}$ factor. 
 
We have the following estimate of our liftings in the weighted
spaces. 
\begin{lem}\label{lem:liftestimate}
  If we have $g\in \overline{\mathfrak{h}}$ then 
  \[ \Vert \mathbf{l}(g) \Vert_{C^{l,\alpha}_{0}} \leq c \Vert
  g\Vert. \]
 Here $\Vert g\Vert$ denotes any fixed norm on the finite dimensional
 vector space $\overline{\mathfrak{h}}$. 
\end{lem}
\begin{proof}
  Using the splitting $\overline{\mathfrak{h}} =
  \overline{\mathfrak{t}}\oplus \mathfrak{h}'$ from
  Definition~\ref{defn:lifting}, if $g\in \mathfrak{h}'$, then
  $\mathbf{l}(g)=g$. On the region where $d > 2\epsilon$, we certainly
  have
  \[ |\nabla^i g| < c \tau^{-i} \vert g\vert \]
  for all $i$, since in fact all the derivatives are bounded
  uniformly, and $\tau \leq 1$.

  On the region where $d < 2\epsilon$ we change coordinates to the $v,
  u_i, w_j'$ from the proof of Proposition~\ref{prop:definesmetric}, 
  and note that in terms of the local
  coordinates $z_i, w_j$ we have
  \[ \begin{aligned}
         \frac{\partial}{\partial v} &= \epsilon
         \frac{\partial}{\partial z_k} + \epsilon \sum_{i=1}^{k-1} u_i
         \frac{\partial}{\partial z_i} \\
         \frac{\partial}{\partial u_i} &= \epsilon
         v\frac{\partial}{\partial z_i} \\
         \frac{\partial}{\partial w_j'} &= \epsilon
         \frac{\partial}{\partial w_j}. 
  \end{aligned} \]
  Using that $|u_i|, |v|< 2$, we obtain the required estimate.

  If $g\in \overline{\mathfrak{t}}$, then $\mathbf{l}(g)$ is defined
  by 
  \[ \mathbf{l}(g) = g + \frac{1}{2} X(A), \]
  where $X = \nabla g$, and 
  \[ A = \epsilon^2 \gamma_2 \big[
\gamma(\epsilon^{-1}d) \log (\epsilon^{-2}d^2) +
\psi(\epsilon^{-2}d^2)\big]. \]
  At the same time, in the coordinates $z, w$ from
  Lemma~\ref{lem:charts}, the vector field $X$ in this case is of the
  form 
  \[ X = a_i \frac{\partial}{\partial z_i} +  b_j
  \frac{\partial}{\partial w_j}, \]
  where $a_i = O(|z|)$ since $X$ is tangent to $S$. Consider the
  region where $d < 2\epsilon$, which is the only place where the
  problematic $\mathrm{log}$ term appears. Here $\gamma_2 = 1$, and note that
  $X(d) = O(|z|) $, while $X(d)/d$ has bounded $z,w$-derivatives. It
  follows that on this region
  \[ \nabla^i X(A) = O(\epsilon^2), \]
  which, combined with the estimate above for $g$ implies the result
  we need. 

  Note that if the vector field $X$ were not parallel to $S$, then
  $X(A)$ would blow up near $S$. This is our reason for lifting
  elements in $\mathfrak{h}'$ in a different way. 
\end{proof}

\subsection{Controlling the inverse}
We will now think of our operator $\widetilde{L}$ as a map between
suitable weighted spaces:
\[ \widetilde{L} : C^{4,\alpha}_\delta(\Bl_S M)^T_0 \times
\overline{\mathfrak{h}} \to C^{0,\alpha}_{\delta-4}(\Bl_SM),  \]
and our goal is the following result. 
\begin{prop}\label{prop:inverse}  
  For sufficiently small $\epsilon > 0$ and $\delta\in (4-2k,0)$, 
  the operator $\widetilde{L}$ is 
  invertible, with a bound on its inverse $P$ independent of $\epsilon$. 
\end{prop}

We prove this result using a blowup argument, following the exposition
in \cite[Theorem 8.14]{GSzBook}. We will need the
following three lemmas. 

\begin{lem}\label{lem:MS}
  Define the linear operator
  \[ \begin{aligned} 
          L_M(u,f) &= L_{\omega_M}(u) - \frac{1}{2} \nabla
          S(\omega_M)\cdot \nabla u - f \\
        &= - \mathcal{D}^*\mathcal{D} u - f. 
\end{aligned} \]
  If $L_M(u,f) = 0$ and $u: M\setminus S\to \mathbf{R}$ is in the
  weighted space $(C^{4,\alpha}_\delta)_0$ with $\delta > 4-2k$, then
  $u, f=0$. 
\end{lem}
\begin{proof}
  The restriction on the weight $\delta$ ensures that $u$ is a
  distributional solution of $L_M(u,f)=0$ on all of $M$, and in
  particular $u$ extends smoothly to $M$. On $M$ we have the equation
  $\mathcal{D}^*\mathcal{D} u + f = 0$ and $\langle u, g\rangle = 0$ for
  all $g\in \overline{\mathfrak{h}}$. This implies that both
  $f$ and $u$ vanish identically. 
\end{proof}

\begin{lem}\label{lem:BlxC}
  If $u : \Bl_0 \mathbf{C}^k \times \mathbf{C}^{n-k} \to \mathbf{R}$
  is in the weighted space $C^{4,\alpha}_\delta$ with $\delta < 0$,
  and $L_0(u) = 0$, then $u=0$. Here $L_0$ denotes the Lichnerowicz operator
  on the product space. 
\end{lem}
\begin{proof}
  We use an argument with the Fourier transform similar to that in
  Mazzeo-Pacard~\cite{MP96} (see also Walpuski~\cite[Lemma
  A.1]{Wal}). 
  Let us write $u(z,w)$, where $z$ denotes the
  coordinate on $\Bl_0\mathbf{C}^k$. We have
  \[ \begin{aligned} L_0(u) &= (\Delta_z + \Delta_w)^2 u + R^{i\bar
      j}u_{i\bar j} \\
  &= \Delta_z^2u + 2\Delta_z\Delta_w u + \Delta_w^2u + R^{i\bar
    j}u_{i\bar j}, 
  \end{aligned} \]
  where $\Delta_z, \Delta_w$ are the Laplacians on the two factors,
  and $R^{i\bar j}$ denotes the Ricci form of $\Bl_0 \mathbf{C}^k$,
  with the indices raised. In particular the $u_{i\bar j}$ terms only
  involve derivatives on the $\Bl_0 \mathbf{C}^k$ factor. 
  
  We take the Fourier transform of
  $u$ in the $w$ variable. This way we obtain a distribution
  $\hat{u}(z,\xi)$ on $\Bl_0\mathbf{C}^k\times \mathbf{C}^{n-k}$,
  satisfying the equation
  \[ \Delta_z^2 \hat{u} - 2|\xi|^2\Delta_z \hat{u} + |\xi|^4\hat{u} +
  R^{i\bar j} \hat{u}_{i\bar j} = 0. \]
  In terms of the Lichnerowicz operator $L_z$ on $\Bl_0\mathbf{C}^k$
  this can be written as
  \[ L_z(\hat{u}) - 2|\xi|^2\Delta_z \hat{u} + |\xi|^4 \hat{u} = 0. \]
  We claim that this implies that the distribution $\hat{u}$ is
  supported on the set $\{\xi =0 \}$. To show this, let $g(z,\xi):
  \Bl_0\mathbf{C}^k\times \mathbf{C}^{n-k}\to\mathbf{R}$ be a smooth
  function with compact support away from $\{\xi =0\}$. We need to
  show that $\langle \hat{u}, g\rangle = 0$. 

  For this we claim that there is a solution $h:\Bl_0\mathbf{C}^k
  \times\mathbf{C}^{n-k} \to \mathbf{R}$ of the equation
  \[ L_z h - 2|\xi|^2 \Delta_z h + |\xi|^4 h = g, \]
  with $h$ decaying faster than any negative power of $|z|$ in the $z$
  direction, and with $h(z,\xi) =0$ if $|\xi|$ is sufficiently large
  (outside the support of $g$). This follows from the fact that for
  any fixed $\xi \ne 0$ the operator
 \[ h \mapsto L_z h - 2|\xi|^2\Delta_z h + |\xi|^4 h \]
  on $\Bl_0\mathbf{C}^k$ is essentially self-adjoint, and has trivial
  kernel in $L^2$. Indeed any solution function $h$ in $L^2$ which is
  in the kernel would have to be
  rapidly decaying by applying Schauder estimates (on balls of radius
  $r/2$ at distance $r$ from the exceptional divisor), and then an
  integration by parts shows that $h=0$. 

  We now have that
  \[ \begin{aligned} \langle \hat{u}, g\rangle &= \langle \hat{u},
    L_zh - 2|\xi|^2\Delta_z h + |\xi|^4h\rangle \\
   &= \langle L_z\hat{u} - 2|\xi|^2\Delta_z \hat{u} + |\xi|^4\hat{u},
   h\rangle \\
    &= 0, \end{aligned} \]
  where the integration by parts is justified since $h$ is rapidly
  decaying in the $z$ direction, and has bounded support in the $\xi$
  direction. 

  Now we know that $\hat{u}$ is supported on $\{\xi = 0\}$, and as a
  result it is a linear combination of derivatives of the delta
  function at the origin in $\xi$ (with coefficients given by
  functions of $z$). In other words we can write
  \[ \hat{u}(z,\xi) = \sum_{i=0}^m u_i(z) \delta^{(i)}(\xi), \]
  where each $\delta^{(i)}$ denotes an $i^{th}$ derivative of the
  delta function at the origin. It follows that
  \[ u(z,w) = \sum_{i=0}^m u_i(z) a_i(w), \]
  where each $a_i(w)$ is an $i^{th}$ degree homogeneous polynomial in
  $w$. Since $u$ is bounded, only a constant polynomial can appear,
  and so we find that $u$ is purely a function of $z$. Since $L_z(u) =
  0$, and we are assuming that $u$ decays at infinity in the $z$
  coordinate, it follows (see e.g. Arezzo-Pacard~\cite{AP06},
  Kovalev-Singer~\cite{KS01}) that 
  $u=0$. 
\end{proof}

\begin{lem}\label{lem:CxC}
  If $u : (\mathbf{C}^k\setminus \{0\}) \times \mathbf{C}^{n-k} \to
  \mathbf{R}$ is in the weighted space $C^{4,\alpha}_\delta$ with
  $\delta\in (4-2k,0)$ and $\Delta^2 u =0$, then $u=0$. 
\end{lem}
\begin{proof}
  Again a local argument shows that actually $u$ extends to a smooth
  function on $\mathbf{C}^k \times \mathbf{C}^{n-k}$, satisfying
  $\Delta^2 u =0$. In addition we know that the function decays at
  infinity in the $\mathbf{C}^k$ factor. An argument identical to that
  in the previous lemma shows that $u=0$. 
\end{proof}

We will now use these results to prove Proposition~\ref{prop:inverse}
\begin{proof}[Proof of Proposition~\ref{prop:inverse}]
  We use an argument by contradiction. We follow the exposition in
  Sz\'ekelyhidi~\cite[Theorem 8.14]{GSzBook} closely. 

  To emphasize the presence of
  the parameter $\epsilon$, we will denote by $\tilde{L}_\epsilon$ our
  operator with respect to the metric $\omega_\epsilon$. Our weighted
  spaces were defined in terms of local coordinates in which
  $\omega_\epsilon$ is uniformly equivalent to the Euclidean
  metric. Using the Schauder estimates in these local charts we obtain
  a uniform constant $C$ (independent of $\epsilon$), such that
  \begin{equation}\label{eq:Schauder}
    \Vert u \Vert_{C^{4,\alpha}_\delta} + \Vert f\Vert\leq C( \Vert
      u\Vert_{C^0_\delta} + \Vert f\Vert + \Vert
      \tilde{L}_\epsilon(u, f)\Vert_{C^{0,\alpha}_{\delta-4}}), 
  \end{equation}
  for all $u, f$ (recall that $\delta\in (4-2k,0)$). 
  We want to show that with a possibly larger constant,
  the same inequality holds without the $\Vert u\Vert_{C^0_\delta}
  +\Vert f\Vert$
  terms on the right hand side, for sufficiently small $\epsilon$. 
  Arguing by contradiction, let us suppose that for a sequence
  $\epsilon_i\to 0$ we have corresponding functions $u_i$ and $f_i$ satisfying
  \[ \Vert u_i\Vert_{C^0_\delta} + \Vert f_i\Vert = 1, \text{ but } \Vert
  \tilde{L}_\epsilon(u_i, f_i)\Vert_{C^{0,\alpha}_{\delta-4}} <
  \frac{1}{i}, \]
  where note that the norms of $u_i$ are computed using
  $\omega_{\epsilon_i}$.  

  Using the equation \eqref{eq:Schauder}
  we can choose a subsequence of the $(u_i, f_i)$, converging to a limit $u :
  M\setminus S\to \mathbf{R}$, in the space $C^{4,\alpha'}_{\delta}$
  with $\alpha' < \alpha$, and $f \in \overline{\mathfrak{h}}$. 
  In particular we find that 
  \[ L_M(u, f) = 0, \]
  where $L_M$ denotes the operator in Lemma~\ref{lem:MS}. From
  Lemma~\ref{lem:MS} we have that $u, f=0$. This
  implies that $f_i \to 0$, and so up to choosing a further subsequence we
  can assume that
  \[   \left\Vert L_\epsilon(u_i) - \frac{1}{2}\nabla
    \mathbf{l}(s)\cdot \nabla u \right\Vert_{C^{0,\alpha}_{\delta-4}} <
  \frac{1}{i}, \]
  i.e. we can drop the term involving $\mathbf{l}(f_i)$ in the definition of
  $\widetilde{L}$.  

  We now need to examine the points $q_i\in \Bl_S M$, where
  $\tau_i^{-\delta} u_i$ achieves its maximum. Recall that $\tau_i$ is
  the weight function on $M$ (or on $\Bl_SM$) we defined before, with
  respect to $\omega_{\epsilon_i}$. By our assumption on $u_i$, we
  have $|\tau_i(q_i)^{-\delta} u_i(q_i)| = 1$. We already know that
  $u_i \to 0$ on compact sets away from $S$, which implies that we
  must have $\tau_i(q_i)\to 0$. We have two cases depending on whether
  $\epsilon_i^{-1} \tau_i(q_i)$ is bounded or not. 

  Suppose first that for some $R > 0$ we have $\epsilon_i^{-1} \tau_i(q_i)
  < R$ for all $i$. For sufficiently large $i$, the points $q_i$ will
  be in charts $z,w$ of the form considered in Lemma~\ref{lem:charts},
  and by changing to $Z_i, W_j$ or $u_i, v, w'$ coordinates as in the proof of
  Proposition~\ref{prop:definesmetric} we can view $q_i$ as a point in
  $\Bl_0\mathbf{C}^k \times \mathbf{C}^{n-k}$, at distance at most $R$
  from $E\times\{0\}$, where $E$ is the exceptional divisor. Moveover
  the pull-backs of $\epsilon_i^{-2}\omega_{\epsilon_i}$ in this chart will converge
  to the product metric on $\Bl_0\mathbf{C}^k\times \mathbf{C}^{n-k}$
  on compact sets. Choosing a further subsequence we can the extract a
  limit $u$ of the functions $\epsilon_i^{-\delta}u_i$, locally in
  $C^{4,\alpha'}_\delta$, where $u : \Bl_0\mathbf{C}^k\times \mathbf{C}^{n-k}\to
  \mathbf{R}$ is in the weighted space $C^{4,\alpha'}_\delta$,
  satisfying $L_0 u =0$. Lemma~\ref{lem:BlxC} implies that $u=0$,
  contradicting that $|\tau_i(q_i)^{-\delta} u_i(q_i)| = 1$. 

  Finally we suppose that $\epsilon_i^{-1}\tau_i(q_i) \to \infty$, but
  $\tau_i(q_i)\to 0$. It follows that in our charts $z,w$ from
  Lemma~\ref{lem:charts} we have $\tau(q_i) = |z|$, up to a bounded
  factor. Arguing as above, by taking a subsequence we can extract a
  limit of the $\tau(q_i)^{-\delta}u_i$, giving a function $u :
  (\mathbf{C}^k\setminus\{0\})\times \mathbf{C}^{n-k} \to \mathbf{R}$,
  in the weighted space $C^{4,\alpha'}_\delta$ satisfying $\Delta^2 u
  =0$. Lemma~\ref{lem:CxC} implies that $u=0$, which is a
  contradiction again. 

  In sum we find that there is a constant $C'$, such that for
  sufficiently small $\epsilon$ the estimate
  \[ \Vert u\Vert_{C^{4,\alpha}_\delta} + \Vert f\Vert \leq C'\Vert
  \tilde{L}_\epsilon(u, f)\Vert_{C^{0,\alpha}_{\delta-4}} \]
  holds. This shows in particular that $\tilde{L}_\epsilon$ has
  trivial kernel, and since it has index zero it must be
  invertible. In addition we obtain the required uniform bound on its
  inverse. 
\end{proof}

\section{The nonlinear equation}\label{sec:nonlinear}
In this section we will solve Equation~\ref{eq:maineq}, which will
then lead to the proof of Proposition~\ref{prop:gluing}. 

\subsection{A better approximate solution}
Our first task is to improve our approximate solution, in a similar
way as was done in Arezzo-Pacard~\cite{AP09}, 
Arezzo-Pacard-Singer~\cite{APS06} and also in
\cite{GSz10}. For this we first 
modify $\omega_M$ on $M\setminus S$, by using a $T$-invariant
solution of
\[ \mathcal{D}^*\mathcal{D} \Gamma = h, \text{ on }M\setminus S, \]
with $h\in \overline{\mathfrak{h}}$ and
\[ \Gamma = -d^{4-2k} + \tilde{\Gamma}, \]
where $d$ is the distance from $S$ as before, and $\tilde{\Gamma} =
O(d^{5-2k})$ (in fact we can have $O(d^{6-2k-\tau})$ for any small
$\tau > 0$, but we will not need this). Such a solution can be
obtained by taking the Green's function for the Lichnerowicz operator
$\mathcal{D}^*\mathcal{D}$, satisfying
\[ \mathcal{D}^*_x\mathcal{D}_x G(x, y) = -\delta_y + h_y \]
for suitable $h_y \in \mathrm{ker}(\mathcal{D})$, and integrating
along $S$: 
\[ \Gamma(x) = c \int_S G(x,y)\, dy, \]
for a suitable constant $c$.
We can ensure that $\Gamma$ is $T$-invariant and $h\in
\overline{\mathfrak{h}}$ by averaging with the $T$-action. 

In a distributional sense $\Gamma$ will then satisfy
\[ \mathcal{D}^*\mathcal{D}\Gamma = h - c_m\delta_S \text{ on }M, \]
where $c_m$ is a dimensional constant and $\delta_S$ is the current of
integration along $S$. Integrating against $g\in\overline{\mathfrak{h}}$
we have
\[ \int_M gh\,\omega^n = c_m\int_S g\,\omega^{n-k}. \]
It follows that under our identifications we have
\[ h = c_m\mu_{\mathscr{S}}(S) + \lambda, \]
where $\lambda = \mathrm{Vol}(M)^{-1}c_m$ is a constant. 

Recall that the potential $\psi$ for the Burns-Simanca metric has an
expansion
\[ \psi(|w|^2) = -|w|^{4-2k} + \tilde{\psi}(|w|^2), \]
 for large $|w|$, where $\tilde{\psi}(|w|^2) = O(|w|^{2-2k})$.
We use this to define a new approximate solution
\[ \begin{aligned}
   \tilde{\omega}_\epsilon = \omega_M + \ddb\Big(
   &-\epsilon^{2k-2}d^{4-2k} + \gamma_1 \epsilon^{2k-2}\tilde{\Gamma}
\\
& + \gamma_2
\epsilon^2\big[\gamma(\epsilon^{-1}d)\log (\epsilon^{-2}d^2) + \tilde{\psi}(\epsilon^{-2}d^2)\big]\Big).
\end{aligned} \]
 Alternatively we can also write
\[ \tilde{\omega}_\epsilon = \omega_\epsilon + \ddb\Big(
\epsilon^{2k-2} \gamma_1\Gamma\Big), \]
in terms of our earlier approximate solution $\omega_\epsilon$. 

We have
\[ \begin{aligned} \Vert \epsilon^{2k-2}\gamma_1
  \Gamma\Vert_{C^{l,\alpha}_2} &\leq c \epsilon^{2k-2} \Vert
  \Gamma\Vert_{C^{l,\alpha}_2(M \setminus N_{r_\epsilon}(S))} \\
  &\leq c \epsilon^{2k-2} r_\epsilon^{2-2k},
\end{aligned} \] 
which tends to zero as $\epsilon\to 0$. It follows that
$\tilde{\omega}_\epsilon$ is a small perturbation of
$\omega_\epsilon$, and so (see
\cite[Proposition 20]{GSz10}) that the linearized operator of
$\tilde{\omega}_\epsilon$ is a small perturbation of the linearized
operator of $\omega_\epsilon$ for sufficiently small $\epsilon$.

\subsection{The non-linear equation}
As in \cite{GSz10}, to prove Proposition~\ref{prop:gluing},
we would like to solve the equation
\begin{equation}\label{eq:main2}  S(\omega_\epsilon + \ddb u) - \frac{1}{2} \nabla\mathbf{l}(f)\cdot
\nabla u = \mathbf{l}(f), 
\end{equation}
with $u, f$ of the form
\[\begin{aligned}
    u &= \epsilon^{2k-2}\gamma_1\Gamma + v
    \\
 f &= S(\omega_M) + \epsilon^{2k-2}h + g,
\end{aligned} \]
where $\Gamma, h$ are defined in the previous subsection. Substituting
these into Equation~\eqref{eq:main2}, we obtain 
\[ \begin{aligned}
  L(v) - \frac{1}{2}X(v) &- \mathbf{l}(g) = \mathbf{l}(S(\omega_M)) -
  S(\omega_\epsilon) - \epsilon^{2k-2} L(\gamma_1\Gamma) -
  Q(u) \\ &+ \frac{1}{2}\nabla\mathbf{l}(\epsilon^{2k-2}h +g)\cdot\nabla u
  + \frac{1}{2}X(\epsilon^{2k-2}\gamma_1\Gamma) +
  \epsilon^{2k-2}\mathbf{l}(h),
\end{aligned} \]
where $X = \nabla\mathbf{l}(S(\omega_M))$, and $L =
L_{\omega_\epsilon}$. 

In the same way as in \cite{GSz10}, using the inverse $P$ of the
linearized operator given by Proposition~\ref{prop:inverse}, we can
rewrite this as a fixed point problem for the operator
\[ \begin{aligned} 
\mathcal{N} : (C^{4,\alpha}_\delta)^T_0 \times
\overline{\mathfrak{h}} &\to (C^{4,\alpha}_\delta)^T_0 \times
\overline{\mathfrak{h}} \\
(v,g) &\mapsto PF(v,g), 
\end{aligned} \]
where
\[ \begin{aligned}
 F(v,g) &= \mathbf{l}(S(\omega_M)) -
  S(\omega_\epsilon) - \epsilon^{2k-2} L(\gamma_1\Gamma) -
  Q(u) \\ &\quad + \frac{1}{2}\nabla\mathbf{l}(\epsilon^{2k-2}h +g)\cdot\nabla u
  + \frac{1}{2}X(\epsilon^{2k-2}\gamma_1\Gamma) +
  \epsilon^{2k-2}\mathbf{l}(h),
\end{aligned}\]
and $u = \epsilon^{2k-2}\gamma_1\Gamma + v$ as above. In addition we
work with a weight $\delta\in (4-2k, 0)$, but very close to
$4-2k$. Following
 the same argument as in \cite{GSz10}, the key estimate that we need is the
following.
\begin{prop}
Choose $\delta \in (4-2k,0)$ very close to $4-2k$. Let
$r_{\epsilon}=e^{\alpha}$ with $\alpha = \frac{2k}{2k+1}$ as before.
Then we have the estimate  
\[ \Vert F(0,0)\Vert_{C^{0,\alpha}_{\delta-4}} \leq  cr_{\epsilon}^{3-\delta}. \]
\end{prop}
\begin{proof}
Let us write $F = F(0,0)$, so we have
\[ \begin{aligned} F &= \mathbf{l}(S(\omega_M)) - S(\omega_\epsilon) - \epsilon^{2k-2} L(\gamma_1\Gamma) -
  Q(\epsilon^{2k-2}\gamma_1\Gamma) \\ &+
  \frac{1}{2}\nabla\mathbf{l}(\epsilon^{2k-2}h) \cdot\nabla (\epsilon^{2k-2}\gamma_1\Gamma)
  + \frac{1}{2}X(\epsilon^{2k-2}\gamma_1\Gamma) +
  \epsilon^{2k-2}\mathbf{l}(h).
\end{aligned} \] 
We will controll $F$ separately on the four regions
$N_{\epsilon}(S)$, $N_{r_\epsilon}(S) \setminus N_{\epsilon}(S)$,
$N_{2r_\epsilon}(S)\setminus N_{r_\epsilon}(S)$, and
$M\setminus N_{2r_\epsilon}(S)$.

\begin{itemize}
\item 
On $N_{\epsilon}(S)$ we have
\[ F = \mathbf{l}(S(\omega_M)) - S(\omega_\epsilon) +
\epsilon^{2k-2}\mathbf{l}(h). \]
Note that 
\[ \Vert \mathbf{l}(S(\omega_M))
\Vert_{C^{0,\alpha}_{\delta-4}(N_{\epsilon})} \leq \epsilon^{4-\delta}
\Vert \mathbf{l}(S(\omega_M))
\Vert_{C^{0,\alpha}_{0}(N_{r_\epsilon})}\leq  c\epsilon^{4-\delta}
\Vert S(\omega_M) \Vert,\]
by Lemma~\ref{lem:liftestimate}. The term involving $h$ is even
smaller. At the same time \eqref{eq:nablaiF}
implies that on $N_\epsilon(S)$ we have  
\[ \Vert S(\omega_{\epsilon})
\Vert_{C^{0,\alpha}_{\delta-4}(N_{\epsilon})} \leq c
\epsilon^{4-\delta}\epsilon^{-1} = c\epsilon^{3-\delta},\]
since $S(\omega_{\epsilon}) = \epsilon^{-2} S(
\epsilon^{-2}\omega_\epsilon)$. In sum we obtain
\[ \Vert F\Vert_{C^{0,\alpha}_{\delta-4}(N_{\epsilon})} < C\epsilon^{3-\delta}. \]

\item On $N_{r_\epsilon}(S)\setminus N_{\epsilon}(S)$ we still have
\[ F = \mathbf{l}(S(\omega_M)) - S(\omega_\epsilon) +
\epsilon^{2k-2}\mathbf{l}(h). \]
As above, from Lemma~\ref{lem:liftestimate} we obtain on this region
that
\[ \Vert \mathbf{l}(S(\omega_M))\Vert \leq  c r_\epsilon^{4-\delta}
\Vert S(\omega_M) \Vert,\]
and we have an even better estimate for $h$. As for the scalar
curvature of $\omega_\epsilon$ we now use \eqref{eq:F1}, to see that
\[ \Vert S(\omega_{\epsilon})
\Vert_{C^{0,\alpha}_{\delta-4}(N_{r_\epsilon}\setminus N_\epsilon)} \leq c
r_\epsilon^{3-\delta},\]
and so 
\[ \Vert F\Vert_{C^{0,\alpha}_{\delta-4}(N_{r_\epsilon}\setminus
  N_\epsilon)} < Cr_\epsilon^{3-\delta}. \]

\item
On $N:=N_{2r_\epsilon}(S)\setminus N_{r_\epsilon}(S),$ we have, as above
\[\Vert \mathbf{l}(S(\omega_M)) \Vert_{C^{0,\alpha}_{\delta-4}}(N)
\leq cr_{\epsilon}^{4-\delta}\Vert \mathbf{l}(S(\omega_M))
\Vert_{C^{0,\alpha}_{0}}\leq cr_{\epsilon}^{4-\delta}.\]
In addition 
\[\Vert \frac{1}{2}\nabla\mathbf{l}(\epsilon^{2k-2}h) \cdot\nabla (\epsilon^{2k-2}\gamma_1\Gamma)
  + \frac{1}{2}X(\epsilon^{2k-2}\gamma_1\Gamma) +
  \epsilon^{2k-2}\mathbf{l}(h)\Vert_{C^{0,\alpha}_{\delta-4}}(N)\leq c
  r_{\epsilon}^{4-\delta},\]
since the largest term here is $\epsilon^{2k-2} X(\gamma_1\Gamma)$,
which contributes 
\[ \epsilon^{2k-2}r_\epsilon^{3-2k} r_\epsilon^{4-\delta} \ll
r_\epsilon^{4-\delta} \]
to the norm. 

So in this region, the main term to control in $F$ is
\[ S(\omega_\epsilon) + L(\epsilon^{2k-2}\gamma_1\Gamma) +
Q(\epsilon^{2k-2}\gamma_1\Gamma) = S(\tilde{\omega}_\epsilon). \]
Using coordinates $(z,w),$ we have
$\tilde{\omega}_\epsilon=\omega_{E}+\ddb{H},$ where 
\[\omega_{E}=\ddb (|z|^2+|w|^2)\]
 is the flat metric in these coordinates and
\[ H =  \phi(z,w) - \epsilon^{2k-2}|z|^{4-2k}( 1+\rho(z,w))^{4-2k} + \epsilon^{2k-2}\gamma_1 \tilde{\Gamma} + \gamma_2
\epsilon^2\tilde{\psi}(\epsilon^{-2}d^2). \] 
Note that 
\[ \nabla^2 H
=O(r_{\epsilon}+\epsilon^{2k-2}r_{\epsilon}^{2-2k}+\epsilon^{2k}r_{\epsilon}^{-2k})
\to 0 \,\, \textrm{as} \,\, \epsilon \to 0.\]
The dominant term above is $\epsilon^{2k-2}r_\epsilon^{2-2k}$, using
our choice of $\alpha$. 
Therefore comparing $\tilde{\omega}_\epsilon$ to the flat metric, we
get that on this region (using Lemma 21 in \cite{GSz10})
\[ \begin{aligned}\Vert S(\tilde{\omega}_\epsilon)-\Delta_{0}^2 
H\Vert_{C^{0,\alpha}_{\delta-4}(N)}&\leq
Cr_{\epsilon}^{4-\delta}\Big( \Vert \nabla^2 H\Vert_{C^{2,\alpha}(N)} \Vert
\nabla^4 H\Vert_{C^{0,\alpha}(N)} \Big)\\
&\leq Cr_{\epsilon}^{4-\delta}\epsilon^{4k-4}r_\epsilon^{2-4k}
\leq Cr_\epsilon^{3-\delta}.\end{aligned} \]
On the other hand $\Delta_{0}^2(|z|^{4-2k})=0.$ Hence
\[\Vert \Delta_{0}^2 H\Vert_{C^{0,\alpha}_{\delta-4}(N)} \leq
Cr_\epsilon^{4-\delta}(1 + \epsilon^{2k-2}r_\epsilon^{1-2k})=
O(r_\epsilon^{3-\delta}), \]
using again our choice of $\alpha$. 

\item
Finally, on $M \setminus N_{2r_\epsilon}(S)$ we have $\omega_\epsilon =
\omega_M$, and 
\[ L(\Gamma) - \frac{1}{2}X(\Gamma) = h, \]
so 
\[ F = -Q(\epsilon^{2k-2}\Gamma) + \frac{1}{2}\nabla
\mathbf{l}(\epsilon^{2k-2}h) \cdot\nabla (\epsilon^{2k-2}\Gamma). \]
Therefore, we get 

\[ \begin{aligned}\Vert F\Vert_{C^{0,\alpha}_{\delta -4}(M\setminus
    N_{2r_\epsilon})} 
  &\leq
  \Vert Q( \epsilon^{2k-2}\Gamma)\Vert_{C^{0,\alpha}_{\delta -4}(M
    \setminus N_{2r_\epsilon})}+c\epsilon^{4k-4} \Vert l(h) \Vert_{C_1^{1,\alpha}} \Vert \nabla
  \Gamma\Vert_{C^{0,\alpha}_{\delta -4}(M \setminus
    N_{2r_\epsilon})}\\
&\leq c\epsilon^{4k-4}\Big(\Vert
  \Gamma\Vert_{C^{4,\alpha}_{\delta }(M \setminus
    N_{2r_\epsilon})}\Vert \Gamma\Vert_{C^{4,\alpha}_{2 }(M \setminus
    N_{2r_\epsilon})}+ r_\epsilon^{-1}\Vert \Gamma\Vert_{C^{1,\alpha}_{\delta-4 }(M
    \setminus N_{2r_\epsilon})}\Big)\\& \leq
  c\epsilon^{4k-4}(r_\epsilon^{4-2k-\delta}r_\epsilon^{2-2k} +
  r_\epsilon^{-1})  \\
&\leq c \epsilon^{4k-4} r_\epsilon^{6-4k-\delta} =
O(r_\epsilon^{3-\delta}),\end{aligned}\]
using the choice of $\alpha$.  
We have used \cite[Proposition 25]{GSz10} and the fact that $\Vert \gamma_{1}\Gamma\Vert_{C^{4,\alpha}_{p }} \leq cr_{\epsilon}^{4-2k-p}$ for $p \geq 4-2k$ and $\Vert \gamma_{1}\Gamma\Vert_{C^{4,\alpha}_{p }} \leq c$ for $p \leq 4-2k.$
\end{itemize}
\end{proof}

The following result has the same proof as Lemma 23 in \cite{GSz10}. 
\begin{lem}\label{lem:contract}
  There is a constant $c_1 > 0$ such that if
 \[ \Vert v_i\Vert_{C^{4,\alpha}_2}, \Vert g_i\Vert
 \leq c_1, \]
  then
  \[ \Vert \mathcal{N}(v_{1}, g_{1}) -
  \mathcal{N}(v_2, g_{2})\Vert_{C^{4,\alpha}_\delta} \leq \frac{1}{2} \Vert
  (v_1,g_{1})-(v_{2},g_2)\Vert_{C^{4,\alpha}_\delta}, \]
  where $\Vert
  (v,g)\Vert_{C^{4,\alpha}_\delta}:=\Vert
  v\Vert_{C^{4,\alpha}_\delta}+\Vert
  g\Vert$
\end{lem}

We now define the open set
\[ \mathcal{U} = \left\{ (v,g) \in (C^{4,\alpha}_\delta)^T_0 \times
\overline{\mathfrak{h}}\,:\, \Vert
  (v,g)\Vert_{C^{4,\alpha}_\delta}  \leq 2Cr_{\epsilon}^{3-\delta} \right\}, \]
where the constant $C$ is the uniform bound on the inverse operators $P$ (c.f. Proposition \ref{prop:inverse}). We then have
\begin{prop}  Suppose $\delta < 0$ is
  sufficiently close to $4-2k$. Then for $\epsilon>0$ sufficiently small, the map 
  $\mathcal{N}:\mathcal{U} \to\mathcal{U}$ is a contraction, and
  therefore has a fixed point $(v_{\epsilon},g_{\epsilon}).$ Moreover, $|g_\epsilon| \leq
  c\epsilon^\kappa$ for some $\kappa > 2k-2$.
\end{prop}
\begin{proof}
 First note that if $(v,g) \in \mathcal{U}$, then we have 
 \[\Vert v\Vert_{C^{4,\alpha}_2}\leq \epsilon^{\delta -2} \Vert v\Vert_{C^{4,\alpha}_{\delta}} \leq 2C\epsilon^{\delta -2}r_{\epsilon}^{3-\delta} \leq c_{1},\] for sufficiently small $\epsilon $. 
 Therefore, Lemma~\ref{lem:contract} applies on  $\mathcal{U}$.  We only need to check $\mathcal{N}(\mathcal{U}) \subset \mathcal{U}$. For this we have
  \[ \begin{aligned} \Vert \mathcal{N}(v,g)\Vert_{C^{4,\alpha}_\delta}
    &\leq \Vert \mathcal{N}(v,g) -
    \mathcal{N}(0,0)\Vert_{C^{4,\alpha}_\delta} + \Vert
    \mathcal{N}(0,0)\Vert_{C^{4,\alpha}_\delta} \\
    &\leq \frac{1}{2}\Vert (v,g)\Vert_{C^{4,\alpha}_\delta} + \Vert
    \mathcal{N}(0,0)\Vert_{C^{4,\alpha}_\delta}\\&\leq Cr_{\epsilon}^{3-\delta}+ C\Vert F(0,0)\Vert_{C^{0,\alpha}_{\delta-4}} \\&\leq 2Cr_{\epsilon}^{3-\delta} . \end{aligned} \]
Therefore we obtain a fixed point $(v_\epsilon, g_\epsilon)$, with
$|g_\epsilon| \leq 2Cr_\epsilon^{3-\delta}$. In addition our choice of
$\alpha$ ensures that when $\delta$ is
sufficiently close to $4-2k$, then $|g_\epsilon| \leq
c\epsilon^{\kappa}$, for some $\kappa > 2k-2$. 
\end{proof}

The fixed point $(v_\epsilon, g_\epsilon)$ gives the solution of
\eqref{eq:main2}, and this completes the proof of Proposition 5.

\subsection*{Acknowledgements}
The authors are grateful to Michael Singer, and Rafe Mazzeo for
helpful suggestions.

\end{document}